\documentclass[11pt]{article}
\usepackage{amsmath,amssymb,amsthm,color}

\usepackage{comment}
\usepackage{color}
\usepackage{amssymb}
\usepackage{amsmath}
\usepackage[english]{babel}
\usepackage{fancyhdr}
\usepackage{amsmath}

\newtheorem{theorem}{Theorem}[section]
\newtheorem{lemma}{Lemma}[section]
\newtheorem{proposition}{Proposition}[section]

\newtheorem{definition}{Definition}[section]

\def\ci{\begin{color}{red}\,}
\def\cf{\end{color}\,}

\begin{document}
\begin{center}
{\Large\bf THE WEDDERBURN B-DECOMPOSITION FOR ALTERNATIVE BARIC ALGEBRAS}

\vspace{.2in}
{\bf Bruno L. M. Ferreira\\
and\\
Ruth Nascimento}
%$^{a}$}
\vspace{.2in}

%$^{a}$
Technological Federal University of Paran\'{a}, Professora Laura Pacheco Bastos Avenue, 800, 85053-510, Guarapuava, Brazil

\vspace{.2in}

brunoferreira@utfpr.edu.br\\
 and\\
ruthnascimento@utfpr.edu.br
\vspace{.2in}

\end{center}

{\bf keywords:} Alternative algebras, Baric algebras, \\
Wedderburn $b$-decomposition. 


\begin{abstract}
In this paper we deal with the Wedderburn $b$-decomposition for alternative baric algebras.
\end{abstract}

\section{Baric algebras}
Baric algebras play a central role in the theory of
genetic algebras. They were introduced by I. M. H. Etherington, in
\cite{Etherington}, in order to give an algebraic treatment to Genetic Populations.
Several classes of baric algebras have been defined,
such as: train, Bernstein, special triangular, etc.

In this paper $F$ is a field of characteristic $\neq 2,3,5$. Let $U$
be an algebra over $F$ not necessarily associative, commutative or
finite dimensional. If $\omega :U\longrightarrow F$ is a nonzero
homomorphism of algebras, then the ordered pair $(U,\omega)$ will
be called a \textit{baric algebra} or {\it b-algebra} over $F$ and $\omega $ its
\textit{weight function} or simply its \textit{weight.} For $x\in
U,$ $\omega (x)$ is called \textit{weight} of $x$.

When $B$ is a subalgebra of $U$ and $B\not\subset ker \omega$,
then $B$ is called a \textit{b-subalgebra} of $(U,\omega).$ In
this case, $(B,\omega_{B})$ is a b-algebra, where
$\omega_{B}=\omega |_{B}:B\longrightarrow F.$  The subset
$bar(B)=\{x\in B~|~ \omega (x)=0\}$ is a two-side ideal of $B$ of
codimension 1, called \textit{bar ideal} of $B.$ For all $b\in B$
with $\omega (b)\neq 0,$ we have $B=Fb\oplus bar(B)$. If $bar(B)$ is a two-side ideal of $bar(U)$
(then by [2, Proposition 1.1], it is also a two-sided
ideal of $U$), then $B$ is called \textit{normal b-subalgebra} of
$(U,\omega)$. If $I\subseteq bar(B)$ is a two-side ideal of $B,$
then $I$ is called \textit{b-ideal} of $B.$

Let $(U,\omega)$ be a b-algebra. A subset $B$ is called
\textit{maximal (normal) b-subalgebra} of $U$ if $B$ is a (normal)
b-subalgebra of $U$ and there is no (normal) b-subalgebra $C$ of
$U$ such that $B\subset C\subset U.$ A subset $I$ is called
\textit{maximal b-ideal} of $U$ if $I$ is a b-ideal of $U,$ $I\neq
bar(U)$ and there is no b-ideal $J$ of $U$ such that $I\subset
J\subset bar(U).$

A nonzero element $e\in U$ is called an {\it idempotent}
if $e^2=e$ and {\it nontrivial idempotent} if it is an idempotent different
from multiplicative identity element. If $(U,\omega)$ is a b-algebra and
$e\in U$ is an idempotent, then $\omega(e)=0$ or  $\omega(e)=1$. When $\omega(e)=1$,
then $e$ is called {\it idempotent of weight 1}.

A b-algebra $(U,\omega)$ is called \textit{b-simple} if for
all normal b-subalgebra $B$ of $U,$ $bar(B)=(0)$ or
$bar(B)=bar(U).$ When $(U,\omega)$ has an idempotent of weight 1,
then $(U,\omega)$ is b-simple if, and only if, its only b-ideals
are $(0)$ and $bar(U).$

Let $(U,\omega)$ be a b-algebra. We define the
\textit{bar-radical} or {\it b-radical} of $U,$ denoted by $rad(U),$ as:
$rad(U)=(0),$ if $(U,\omega)$ is b-simple, otherwise as
$rad(U)=\bigcap bar(B),$ where $B$ runs over the maximal normal
b-subalgebra of $U.$ Of course, $rad(U)$ is a b-ideal of $U.$

We say that $U$ is \textit{b-semisimple} if $rad(U)=(0).$

\section{Alternative algebras}
In this section, we present some definitions and properties of alternative algebras and prove some results which will be used later.
\vspace{.2in}

An algebra $U$ over a field $F$ is called {\it alternative algebra } if it satisfies the identities:
\begin{eqnarray}
(x,x,y) = (y,x,x) = 0,
\end{eqnarray}
for all $x,y \in U$, where the $(x,y,z)=(xy)z-x(yz)$ is the {\it associator of} the elements $x,y,z$.

Let $U$ be an alternative algebra over $F$. Then, $U$ is a power-associative algebra and if $U$ has an idempotent $e$, then $U$ is the vector space direct sum $U=U_{11}\oplus U_{10}\oplus U_{01}\oplus U_{00}$, where \\
\centerline{$U_{ij}=\{x_{ij}\in U~|~ex_{ij}=ix_{ij}~\textrm{and}~ x_{ij}e=jx_{ij}\}$ $(i,j=0,1)$}\\ satisfying the multiplicative relations $U_{ij}U_{jl}\subset U_{il}$, $U_{ij}U_{ij} \subset U_{ji}$ and $U_{ij}U_{kl} = 0$, \ $j \neq k$, \ \ $(i,j,l=0,1)$, see  \cite{Sch}.

A set of idempotents $\{e_{1}, \ldots , e_{t}\}$, in an (arbitrary) alternative algebra, is called {\it pairwise orthogonal} in case $e_{i}e_{j}=0$ for $i\neq j$. Note that any sum $e = e_{1}+ \cdots +e_{t}$, of pairwise orthogonal idempotents $(t\geq 1)$, is an idempotent. Also, $ee_{i}=e_{i}e=e_{i}$, $(i = 1, \ldots,t)$.

A more refined Peirce decomposition for an alternative algebra than the one given above is the following decomposition relative to a set $\{e_{1}, \ldots , e_{t}\}$, of pairwise orthogonal idempotents in $U$: $U$ is the vector space direct sum\\
\centerline{$U=\bigoplus_{i,j} U_{ij}~(i,j = 0,1, \ldots,t)$,}\\
where $U_{ij}=\{x_{ij}\in U~|~e_{k}x_{ij}=\delta_{ki} x_{ij}~\textrm{and}~ x_{ij}e_{k}=\delta_{jk}x_{ij}~\textrm{for}~(k = 1,\ldots, t)\}$ $(i,j = 0,1, \ldots,t)$, satisfying the multiplicative relations:
\begin{eqnarray}\label{subspace}
&&U_{ij}U_{jl}\subset U_{il}~(i,j,l=0,1, \ldots,t),\\
&&U_{ij}U_{ij}\subset U_{ji}~(i,j=0,1, \ldots,t),\\
&&U_{ij}U_{kl} = 0 ~ \ \ j \neq k, \ \  (i,j)\neq(k,l) \ \ \ (i,j,k,l=0,1, \ldots,t),
\end{eqnarray}
where $\delta _{jk}~(j,k= 0,1, \ldots,t)$ is the {\it Kronecker delta}.

An nonzero ideal $I$ of an alternative algebra $U$ is called {\it minimal} if for any ideal of $U$ such that $(0)\subset J\subset I$, then $J=(0)$ or $J=I$.

Let $U$ be a finite dimensional alternative algebra over $F$, since $U$ is a power-associative algebra, then by \cite{Sch} $U$ has a unique maximal nilideal, we define {\it nilradical} $R(U)$ of $U$ as the maximal nil ideal of $U$. Let us say that $U$ is {\it simple} when its only ideals are the trivial ideals and $U$ is not a zero algebra. If $R(U)=0$, then $U$ is called {\it semisimple}.

\begin{lemma}\label{l0} Let $U$ be a finite dimensional alternative algebra over $F$ with a non trivial idempotent $e$. If $U=\bigoplus_{i,j}U_{ij}~(i,j=0,1),$ relative to $e$, then
\begin{eqnarray*} R(U_{ii})=R(U)\cap U_{ii}~(i=0,1).
\end{eqnarray*}
\end{lemma}
\begin{proof} See \cite[Corollary 3.8]{Sch} .\end{proof}

\begin{proposition}\label{p21} Let $U$ be a finite dimensional alternative algebra. If $I$ is a mi\-nimal ideal of $U$, then either $I^{2}=0$ or $I$ is simple.
\end{proposition}
\begin{proof}\cite[Chap. VIII, Theorem 10]{ZHEV}.\end{proof}

\section{Baric alternative algebra}

In this section, we introduce a notion of Wedderburn b-decomposition of a $b$-alternative algebra and we present conditions for which it has such decomposition.

If $(U,\omega)$ is a b-algebra and $I$ is a b-ideal of $U$, then $\left(U/I,\bar{\omega}\right)$ is a b-algebra, where $\bar{\omega}(u+I)=\omega(u)$.

\begin{definition} Let $(U,\omega)$ be $b$-alternative algebra over a field $F$. We say that $U$ has a Wedderburn $b$-decomposition if we can decompose $U$ as a direct sum $U=S\oplus V\oplus rad(U)$ (vector space direct sum), where $S$ is a b-semisimple b-subalgebra of $U$ and $V$ is a vector subspace of $bar(U)$ such that $V^2\subset rad(U)$.
\end{definition}

\begin{lemma}\label{l1} Let $U$ be a finite dimensional b-alternative algebra over $F$ with unity element $1$ and $I$ a $b$-ideal of $U$ such that $I\subset R(U)$. If $\overline{u_{1}}$ is a nonzero idempotent of $bar\big(U/I\big)$, then there is an idempotent $e_{1}$ in $bar(U)$ such that $\overline{e_{1}}=\overline{u_{1}}.$ Moreover, if $bar(U)$ is an algebra with a unity $f$ and $\overline{f}=\overline{u_{1}}$, then $f=e_{1}$.
\end{lemma}
\begin{proof} Let us consider the quotient algebra $U/I=\{\overline{x}~|~x\in U\}$ and the application $\overline{\omega}:U/I\rightarrow F$ defined by $\overline{\omega}(\overline{x})=\omega (x)$, for all $x\in U$. Then $\overline{\omega}$ is a nonzero algebra homomorphism and therefore $(U/I,\overline{\omega})$ is a b-algebra such that $U/I=F\overline{1}\oplus bar\big(U/I)$, where $bar\big(U/I)=bar(U)/I.$

Next, since $\overline{u_{1}}$ is an idempotent of $bar\big(U/I\big)$, then any representative $u_{1}$ of $\overline{u_{1}}$ is non nilpotent and belongs to $bar(U)$. It follows that the subalgebra generated by the element $u_{1}$ is a non nil subalgebra of $bar(U)$. This implies that $bar(U)$ has an idempotent $e_{1}=\sum_{i} \alpha _{i}u_{1}^{i}$, $\alpha _{i}\in F,$ verifying $\overline{e_{1}}=\alpha \overline{u_{1}},$ $\alpha \in F,$ $\alpha =\sum_{i} \alpha _{i}.$ Since $e_{1}\notin
R(U)$, then $e_{1}\notin I$ and it follows that $\overline{e_{1}}\neq \overline{0}$ and $\overline{e_{1}}=\alpha \overline{e_{1}}$. Hence $\alpha =1$ and $\overline{e_{1}}=\overline{u_{1}}.$ Moreover, if $bar(U)$ is an algebra with a multiplicative unity $f$ and $\overline{f}=\overline{u_{1}}$, then $\overline{f}=\overline{e_{1}}$ which implies $f-e_{1}\in I$. Since $(f-e_{1})^{2}=f-e_{1}$, then $f=e_{1}$.\end{proof}

\begin{lemma}\label{l32} Let $U$ be a finite dimensional b- alternative algebra over $F$ with unity element $1$ and $J$ a $b$-ideal of $U$ such that $J\subset R(U)$. If $\{\overline{u_{1}}, \ldots ,\overline{u_{t}}\}$ is a set of nonzero pairwise orthogonal idempotents of $bar\big(U/J\big)$, then there are a set of nonzero pairwise orthogonal idempotents $\{e_{1}, \ldots ,e_{t}\}$ of $bar(U)$ veri\-fying $\overline{e_{i}}=\overline{u_{i}}~(i=1, \ldots ,t).$ Moreover, if $e$ is any idempotent of $bar(U)$ such that $\overline{e}=\sum_{i=1}^{t} \overline{u_{i}}$, we may choose by $e_{i}$ such that $e=\sum_{i=1}^{t} e_{i}.$
\end{lemma}
\begin{proof} To prove this lemma we use the principle of mathematical induction. For $t=1$, the result is true, by Lemma \ref{l1}. Now, suppose that for a positive integer $t\geq 1$, the lemma is true. Then for the set of nonzero pairwise orthogonal idempotents $\{\overline{u_{1}}, \ldots ,\overline{u_{t+1}}\}$ of $bar\big(U/J\big),$ there is a set of nonzero pairwise orthogonal idempotents $\{e_{1}, \ldots , e_{t}\}$, of $bar(U)$, verifying $\overline{e_{i}}=\overline{u_{i}}~(i=1, \ldots ,t),$ by the principle of mathematical induction. Let us consider the Peirce decompositions:
$$bar(U)=\bigoplus_{i,j}bar(U)_{ij} \hbox{ and } bar\big(U/J\big)=\bigoplus_{i,j}bar\big(U/J\big)_{ij}~(i,j=0,1),$$
relative to idempotents $e=\sum_{i=1}^{t} e_{i}$ and  $\overline{e}=\sum_{i=1}^{t} \overline{e_{i}},$ respectively. It follows that: 
\begin{enumerate}
\item[(i)] $e_{i}\in bar(U)_{11}$ $(i=1, \ldots , t)$; 
\item[(ii)] $\overline{e_{i}}\in bar\big(U/J\big)_{11}$ $(i=1, \ldots , t)$; 
\item[(iii)] $\overline{u_{t+1}}\in bar\big(U/J\big)_{00}$.
\end{enumerate}

Let us define $P=F1\oplus bar(U)_{00}$. Then $(P,\omega_{P})$ is a finite dimensional b-subalgebra of $(U,\omega)$ with unity element $1$ and $bar(P)=bar(U)_{00}$, where $\omega _{P}:=\omega |_{P}$. Let us define $K=J\cap bar(U)_{00}$. It is easy to check that $K$ is a $b$-ideal of $P$ such that $K\subset R(P)$, because $R(P)=R\big(bar(P)\big)=R\big(bar(U)_{00}\big)$, by \cite[Proposition 4.1]{Couto} and the fact that $R\big(bar(U)_{00}\big)=R\big(bar(U)\big)\cap bar(U)_{00}=R(U)\cap bar(U)_{00}$, according to Lemma \ref{l0}, where $R\big(bar(U)\big)=R(U)$, again by \cite[Proposition 4.1]{Couto}.

Let us consider the quotient algebra $P/K=\{\widetilde{x}~|~x\in P\}$ and $\widetilde{\omega}:P/K\rightarrow F$, defined by $\widetilde{\omega_{P}}(\widetilde{x}):=\omega_{P}(x)$ for all $x\in P.$ Certainly, $(P/K, \widetilde{\omega_{P}})$ is a b-algebra such that $bar(P/K)=bar(P)/K=bar(U)_{00}/\big(J\cap bar(U)_{00}\big).$

Also observe that
\begin{eqnarray}\label{iso}
bar\big(U/J\big)_{00}=\big(bar(U)_{00}+J\big)/J\cong bar(U)_{00}/\big(J\cap bar(U)_{00}\big).
\end{eqnarray}

Since $\overline{u_{t+1}}\in bar\big(U/J\big)_{00}$, we can assume that $u_{t+1}\in bar(U)_{00}.$ In fact, let us write $u_{t+1}=a_{11}+a_{10}+a_{01}+a_{00}$, where $a_{ij}\in bar(U)_{ij}$ $(i,j=0,1).$ Then $\overline{u_{t+1}}=\overline{a_{11}}+\overline{a_{10}}+\overline{a_{01}}+\overline{a_{00}}$ which implies  $\overline{0}=\overline{e}\,\overline{u_{t+1}}=\overline{a_{11}}+\overline{a_{10}}$ and  $\overline{0}=\overline{u_{t+1}}\,\overline{e}=\overline{a_{11}}+\overline{a_{01}}$. Thus, $\overline{a_{11}}=\overline{a_{10}}=\overline{a_{01}}=\overline{0}$ implying $\overline{u_{t+1}}=\overline{a_{00}}$.

From the isomorphism, in (\ref{iso}), and the assumption on the element idempotent $\overline{u_{t+1}}$, we have $\widetilde{u_{t+1}}\in bar(P/K)$ which implies that there is an idempotent $e_{t+1}\in bar(P)$ such that $\widetilde{e_{t+1}}=\widetilde{u_{t+1}},$  by Lemma \ref{l1}. From the isomorphism, in (\ref{iso}), we conclude that $e_{t+1}\in bar(U)$ and $\overline{e_{t+1}}=\overline{u_{t+1}}.$ Since $e_{t+1}\in bar(U)_{00},$ then the elements idempotent $e_{1},
\ldots , e_{t},e_{t+1}$ are pairwise orthogonal.

Finally, suppose that $e$ is an arbitrary idempotent of $bar(U)$ such that $\overline{e}=\sum_{i=1}^{t} \overline{u_{i}}$. Let us consider the Peirce decompositions $bar(U)=\bigoplus_{i,j}bar(U)_{ij}$ and $bar\big(U/J\big)=\bigoplus_{i,j}bar\big(U/J\big)_{ij}~(i,j=0,1)$, relative to idempotents $e$ and $\overline{e}$, respectively. It follows that $\overline{u_{i}}\in bar\big(U/J\big)_{11}$ $(i=1, \ldots , t)$.

Let us define the vector subspace of $Q=F1\oplus bar(U)_{11}$ of $U$. Naturally, $Q$ is an subalgebra of $U$ such that $Q\not\subset \ker(\omega)$. It follows that $(Q,\omega_{Q})$, where $\omega _{Q}:=\omega |_{Q}$, is a finite dimensional $b$-subalgebra of $U$ with unity element $1$ and $bar(Q)=bar(U)_{11}$. Let us define $L=J\cap bar(U)_{11}$. As in the previous definitions, certainly $L$ is a $b$-ideal of $Q$ such that $L\subset R(Q)$, because $R(Q)=R\big(bar(Q)\big)=R\big(bar(U)_{11}\big)$, by \cite[Proposition 4.1]{Couto} and the fact that $R\big(bar(U)_{11}\big)=R\big(bar(U)\big)\cap bar(U)_{11}=R(U)\cap bar(U)_{11}$, according to Lemma \ref{l0}, where $R\big(bar(U)\big)=R(U)$, again by \cite[Proposition 4.1]{Couto}.

Let us take the quotient algebra $Q/L=\{\widetilde{x}~|~x\in Q\}$ and $\widetilde{\omega}:Q/L\rightarrow F$, defined by
$\widetilde{\omega_{Q}}(\widetilde{x}):=\omega_{Q}(x)$ for all $x\in Q.$ Again, we have that $(Q/L, \widetilde{\omega_{Q}})$ is a b-algebra such that $bar(Q/L)=bar(Q)/L=bar(U)_{11}/\big(J\cap bar(U)_{11}\big).$

Now, let us observe that
\begin{eqnarray}\label{iso2}
bar\big(U/J\big)_{11}=\big(bar(U)_{11}+J\big)/J\cong bar(U)_{11}/\big(J\cap bar(U)_{11}\big).
\end{eqnarray}

Since $\overline{u_{i}}\in bar\big(U/J\big)_{11}$ and $\overline{u_{i}}=\overline{e}\,\overline{u_{i}}\,\overline{e}$ $(i=1, \ldots , t)$, we can take a representative $u_{i}$ of $\overline{u_{i}}$ in $bar(U)_{11}$ $(i=1, \ldots , t)$.

From the isomorphism, in(\ref{iso2}), and the assumption on the element idempotent $\overline{u_{i}}$, we have that $\widetilde{u_{i}}\in bar(Q/L)$ which implies that there is a set of idempotents $e_{1}, \ldots ,e_{t}$, in $bar(Q)$, pairwise orthogonal, such that $\widetilde{e_{i}}=\widetilde{u_{i}}$ $(i=1, \ldots , t)$.

As the idempotent $e$ is a multiplicative unity in the subalgebra $bar(U)_{11}$ and $\overline{e}=\sum_{i=1}^{t} \overline{e_{i}},$ then $e-\sum_{i=1}^{t} e_{i}\in J\subset R(U)$ and $\big(e-\sum_{i=1}^{t} e_{i}\big)^{2}=e-\sum_{i=1}^{t} e_{i}$ which implies $e=\sum_{i=1}^{t} e_{i}.$\end{proof}

%\begin{definition}
%Let $U$ be a b-algebra of $(\gamma, \delta)$ type over a field $F$. We say that $bar(U)$ %contains a total matrix algebra if it contains a subalgebra isomorphic to an algebra of matrix %$\mathfrak{M}_{t}(F)$ for some $t$. We indicate this total matrix algebra by $\mathfrak{M}_{t}$ %(or simply $\mathfrak{M}$).
%\end{definition}

\begin{lemma}\label{l33} Let $(U,\omega)$ be a finite dimensional b-algebra of $(\gamma, \delta)$ type with unity element $1$ and $J$ a $b$-ideal of $U$ such that $J\subset R(U)$. If $bar(U/J)$ contains a total matrix algebra $\mathfrak{M}_{t}$ of degree $t$ with identity element $\overline{u}$ and $f$ is an idempotent of $bar(U)$ such that  $\overline{f}=\overline{u}$, then $bar(U)$ contains a total matrix algebra $\mathfrak{M}$ of degree $t$ with identity element $f$ such that $\overline{\mathfrak{M}}=\mathfrak{M}_{t}.$
\end{lemma}
\begin{proof} Let $\mathfrak{M}_{t}$ be a total matrix algebra $\mathfrak{M}_{t}$ of degree $t$ with identity element $\overline{u}$. By hypothesis we have $\mathfrak{M}_{t}=\{\overline{u_{ij}}~|~i,j=1, \ldots ,t\}$, with the familiar multiplication table $\overline{u_{ij}}\ \overline{u_{kl}}=\delta _{jk}\overline{u_{il}}$ $(i,j=1, \ldots ,t)$.

By Lemma \ref{l32}, there exist pairwise orthogonal idempotents $f_{11}, \ldots ,f_{tt}$, in $bar(U)$, such that $\overline{f_{ii}}=\overline{u_{ii}}~(i=1, \ldots ,t)$ and $f=\sum_{i=1}^{t} f_{ii}.$

Let us consider the Peirce decompositions $bar(U)=\bigoplus_{i,j} bar(U)_{ij}$ and $bar\big(U/J\big)=\bigoplus_{i,j} bar\big(U/J\big)_{ij}~(i,j=0,1, \ldots ,t)$, relative to the sets of idempotents  $\{f_{11}, \ldots ,f_{tt}\}$ and $\{\overline{f_{11}}, \ldots ,\overline{f_{tt}}\}$, respectively. It follows that: {\it (i)} $f_{ii}\in bar(U)_{ii}$ $(i=1, \ldots , t)$; and {\it (ii)} $\overline{f_{ii}}\in bar\big(U/J\big)_{ii}$ $(i=1, \ldots , t)$.

Now, let us observe that for every index $i=2, \ldots ,t$, we can take the representative $u_{i1}$, of $\overline{u_{i1}}\in \mathfrak{M}_{t}$, in $bar(U)_{i1}.$ For a $i=1$, let us take $u_{11}=f_{11}.$ Similarly, for every index $j=2, \ldots ,t$, we can take the representative $u_{1j}$, of $\overline{u_{1j}}\in \mathfrak{M}$, in $bar(U)_{1j}.$

Yet, since $\overline{u_{1j}}\,\overline{u_{j1}}=\overline{f_{11}}$ $(j=1, \ldots , t)$, then $u_{1j}u_{j1}=f_{11}+a_{j},$ where $a_{j}\in J\cap bar(U)_{11}$ is a nilpotent element. Let us consider $m$ the smallest positive integer such that $a_{j}^{m}=0$ and let us define $b_{j}=\sum _{i=1}^{m-1} (-a_{j})^{i}.$ Then: {\it (i)} $b_{j}\in J\cap bar(U)_{11}$; {\it (ii)} $b_{j}a_{j}=-\sum _{i=2}^{m-1} (-a_{j})^{i}$; and {\it (iii)} $a_{j}+b_{j}+b_{j}a_{j}=0$. It follows that $(f_{11}+b_{j})(f_{11}+a_{j})=f_{11}+a_{j}+b_{j}+b_{j}a_{j}=f_{11}.$

Let us define $f_{i1} = u_{i1}$ and $f_{1j} = (f_{11} + b_{j})u_{1j}$ $(i,j = 2,\ldots,t)$. Then $f_{1j}f_{j1} = \big((f_{11} + b_{j})u_{1j}\big)u_{j1} = (f_{11}u_{1j})u_{j1} + (b_{j}u_{1j})u_{j1} = f_{11}(u_{1j}u_{j1}) + b_{j}(u_{1j}u_{j1})= f_{11}(f_{11}+a_{j}) + b_{j}(f_{11}+a_{j}) = f_{11}$. 
% Podemos sempre associar em algebras alternativas (usando a decomposiçao de Peirce mais refinada) (u_{ii}, b_{ij}, c_{ji}) j \neq i. Demo (u_{ii} + b_{ij}, u_{ii} + b_{ij}, c_{ji}) = 0 pela lei alternativa. 
Next, let us define $f_{ij} = f_{i1}f_{1j}$ $(i \neq j; i,j = 2,\ldots,t)$. From a direct calculus, we have $\overline{f_{ij}} = \overline{u_{ij}}$ and $f_{ij}f_{kl} =\delta _{jk} f_{il}$ $(i,j,k,l=1, \ldots , t)$. Thus, the set $\{f_{ij}~|~i,j = 1, \ldots, t\}$ is a basis for a total matrix algebra $\mathfrak{M}$ of degree $t$, in $bar(U)$, with identity element $f$ such that $\overline{\mathfrak{M}}=\mathfrak{M}_{t}$.\end{proof}

\begin{lemma}\label{l34} Let $(U,\omega)$ be a finite dimensional $b-$alternative algebra with unity element $1$ and $J$ a $b$-ideal of $U$ such that $J\subset R(U)$. If $bar(U/J)$ contains a direct sum of $b$-ideals $\mathfrak{M}_{t_{1}} \oplus \cdots \oplus \mathfrak{M}_{t_{s}}$, where each $\mathfrak{M}_{t_{i}}$ is a total matrix algebra of degree $t_{i}$ $(i=1, \ldots ,s)$, then $bar(U)$ contains a direct sum of pairwise orthogonal subalgebras $\mathfrak{M}_{1} \oplus \cdots \oplus \mathfrak{M}_{s}$, where each $\mathfrak{M}_{i}$ is a total matrix algebra of degree $t_{i}$ $(i=1, \ldots ,s)$, such that $\overline{\mathfrak{M}_{i}}=\mathfrak{M}_{t_{i}}$ and \begin{eqnarray*}
\mathfrak{M}_{1} \oplus \cdots \oplus \mathfrak{M}_{s}\cong \mathfrak{M}_{t_{1}} \oplus \cdots \oplus \mathfrak{M}_{t_{s}}.
\end{eqnarray*}
\end{lemma}
\begin{proof} Let $\overline{e_{t_{i}}}$ be the unity element of $\mathfrak{M}_{t_{i}}$ $(i = 1, \ldots, s)$. By Lemma \ref{l32}, $bar(U)$ has a set of idempotents $e_{1}, \ldots ,e_{s}$, pairwise orthogonal, such that $\overline{e_{i}}=\overline{e_{t_{i}}}$ $(i=1, \ldots , s)$. This implies that $bar(U)$ contains a total matrix algebra $\mathfrak{M}_{i}$ of degree $t_{i}$ with identity element $e_{i}$ such that $\overline{\mathfrak{M}_{i}}=\mathfrak{M}_{t_{i}}$, by Lemma \ref{l33}.

Let us consider the Peirce decomposition $bar(U)=\bigoplus_{i,j} bar(U)_{ij}$ $(i,j=1, \ldots , s)$, relative to set of idempotents $\{e_{1}, \ldots , e_{s}\}$. For all element $x_{i}\in \mathfrak{M}_{i}$ $(i=1, \ldots , s)$, we have $x_{i}=e_{i}x_{i}$. But in an alternative algebra each associator $(x, e_j, e_l) = 0$ and $(e_j, e_l, x) = 0$ ($j, l = 1, \cdots , s $),  which implies $e_{k}x_{i}=e_{k}(e_{i}x_{i})
=(e_{k}e_{i})x_{i}=\delta _{ki} x_{i}$. Similarly, we show $x_{i}e_{k}=\delta _{ik}x_{i}$. Thus, $\mathfrak{M}_{i}\subset bar(U)_{ii}$ $(i=1, \ldots ,s)$. Since the subalgebras $bar(U)_{ii}$ $(i=1, \ldots , s)$ are pairwise orthogonal, then the sum $\mathfrak{M}_{1} \oplus \cdots \oplus \mathfrak{M}_{s}$, is a direct sum, pairwise orthogonal, such that
$\mathfrak{M}_{1} \oplus \cdots \oplus \mathfrak{M}_{s}\cong \mathfrak{M}_{t_{1}} \oplus \cdots \oplus \mathfrak{M}_{t_{s}}$.\end{proof}

\begin{lemma}\label{l35} Let $(U,\omega)$ be a finite dimensional $b-$alternative algebra with unity element $1$ and $J$ a $b$-ideal of $U$ such that $J\subset R(U)$. If $bar(U/J)$ contains a direct sum of $b$-ideals $J_{1} \oplus \cdots \oplus J_{r}$ such that $J^2_{i} = \overline{0}$ $(i= 1, \ldots, r)$, then $bar(U)$ contains a vector subspace $V$ such that $V\cong \overline{V}=J_{1} \oplus \cdots \oplus J_{r}$ and $V^2 \subset rad(U)$.
\end{lemma}
\begin{proof} Let $\big\{\overline{v_{i,1}},\overline{v_{i,2}}, \ldots, \overline{v_{i,r_{i}}}\big\}$ be a basis of vector subspace $J_{i}$ $(i = 1, \ldots, r)$ and $v_{i,j}$ a representative of the class $\overline{v_{i,j}}$ $(i = 1, \ldots, r;~j= 1, \ldots ,r_{i})$, in $bar(U)$. From a direct calculus, we have: {\it (i)} the set $\bigcup\limits_{i=1}^{r} \big\{v_{i,1},v_{i,2}, \ldots, v_{i,r_{i}}\big\}$ is linearly independent; and {\it (ii)} $v_{i,j}v_{k,l} \in rad(U)$ $(i,k = 1, \ldots, r)$ and $(j = 1, \ldots, r_{i};~l= 1, \ldots, r_{k})$.

Let us define $V$ the vector subspace generated by the set\\
\centerline{$\bigcup\limits_{i=1}^{r} \big\{v_{i,1},v_{i,2}, \ldots, v_{i,r_{i}}\big\}$.}\\
It follows that  $V\cong \overline{V}=J_{1} \oplus \cdots \oplus J_{r}$ and $V^2 \subset rad(U)$.\end{proof}

\begin{lemma}\label{lnovo} Let $(U,\omega)$ be a finite dimensional $b$-alternative algebra with unity element $1$ and $J$ a $b$-ideal of $U$ such that $J^2 = 0$. If $bar(U/J)$ contains a direct sum of $b$-ideals $I_{1} \oplus \cdots \oplus I_{r}$ such that $I_{i}$ is a split Cayley algebra $(i= 1, \ldots, r)$, then $bar(U)$ contains a subalgebra $\mathcal{C} \cong I_{1} \oplus \cdots \oplus I_{r}$.
\end{lemma}
\begin{proof}
We may take $I_k = F_2 + \overline{w_k}F_2$, $\overline{w_k}^2 = \overline{1}$ by \cite[Lemma 3.16]{Sch} where $F_2$ is the algebra of all $2\times 2$ matrices over $F$, $k= 1, \cdots , r$. By Lemma \ref{l34}, $bar(U)$ contains a total matrix algebra $\mathcal{D} \cong F_2$ such that $\mathcal{D}$ contains an identity element and the matric basis $\left\{e_{ij}\right\}$ of $\mathcal{D}$ yields the matric basis $\left\{\overline{e_{ij}}\right\}$ of $F_2$. 
Let $\iota: \mathcal{D}\longrightarrow \mathcal{D}$ the involution in $\mathcal{D}$. We know $x + \iota(x) = t(x)1$ for all $x \in \mathcal{D}$, where $t(x)$ is the trace of $x$ and $1$ is the identity element of $\mathcal{D}$ by \cite[page 45]{Sch}. Note that $\iota(\overline{x}) = \overline{\iota(x)}$, $a(\overline{w_k} b) = \overline{w_k}(\iota(a)b)$, $(\overline{w_k} a)b = \overline{w_k}(ba)$ and $(\overline{w_k} a)(\overline{w_k} b) = b\iota(a)$ for $x, a, b \in \mathcal{D}$, we have indicated to the reader \cite[Chapter III, Sec. 4]{Sch} for Cayley Algebras.
In order to prove the lemma, it is sufficient to show the existence of $v_k \notin \mathcal{D}$ satisfying $v^2_k = 1$, $\overline{v_k} = \overline{w_k}$ and $xv_k = v_k \iota(x)$ for all $x \in \mathcal{D}$.   

Consider $\overline{f_{ij}} = \overline{w_k} \  \overline{e_{jj}}$ for $i \neq j $ $(i, j = 1,2).$

Using the Peirce decomposition of $bar(U)$ relative to $e_1 = e_{11}$, $e_2 = e_{22}$, we may take $f_{ij} \in bar(U)_{ij}$ $(i \neq j)$. In fact $\overline{e_{ii}} (\overline{f_{ij}}\overline{e_{jj}}) = \overline{e_{ii}} (\overline{w_k} \ \overline{e_{jj}}^2) =\overline{w_k} \ \overline{\iota(e_{ii})e_{jj}} = \overline{w_k} \ \overline{e_{jj}}  = \overline{f_{ij}}$.
Now $\overline{e_{ji}} \ \overline{f_{ij}} = \overline{e_{ji}} \ (\overline{w_k} \ \overline{e_{ij}}) = - \overline{w_k}(\overline{e_{ji}} \ \overline{e_{jj}}) = 0$, implying that
$$e_{ji}f_{ij} = c_j, \ \ \ \ c_j \in J\cap bar(U)_{jj} \ \ \ \ \ (i \neq j; i,j = 1,2).$$

Write $h_{ij} = f_{ij} - e_{ij}c_j$. Then $h_{ij} \in bar(U)_{ij}$, $\overline{h_{ij}} = \overline{f_{ij}}$, and 
$$e_{ji}h_{ij} = h_{ij}e_{ji} = 0  \ \ \ \ \ \ (i \neq j; i,j = 1, 2).$$
In fact by Lemma \ref{l33} we know $e_{ji}e_{ij} = e_{jj}$, so $e_{ji}h_{ij} = c_j - e_{ji}(e_{ij}c_j) = c_{j} - (e_{ji}e_{ij})c_j = 0$. Also $e_{ij}c_j = e_{ij}(e_{ji}f_{ij}) = (e_{ij}e_{ji})f_{ij} - (e_{ij}, e_{ji}, f_{ij}) = f_{ij} + (f_{ij}, e_{ji}, e_{ij}) = f_{ij} + (f_{ij}e_{ji})e_{ij} - f_{ij} = (f_{ij}e_{ji})e_{ij}$, so that 
$$h_{ij}e_{ji} = f_{ij}e_{ji} - [(f_{ij}e_{ji})e_{ij}]e_{ji} = 0.$$
Now $\overline{h_{ij}} \ \overline{h_{ji}} = \overline{f_{ij}} \ \overline{f_{ji}} = \overline{e_{ii}} \ \overline{\iota(e_{jj})} = \overline{e_{ii}}$ implies that 
$$h_{ij}h_{ji} = e_{ii} + a_{i}, \ \ \ \ \ a_{i} \in J \cap bar(U)_{ii} \ \ \ \ \ \ \ (i \neq j; i,j = 1,2).$$
Then $a_{i}^2 = 0$ since $J^2 = 0$, and 
$$(e_{ii} - a_{i})(e_{ii} + a_i) = e_{ii} = (e_{ii} + a_{i})(e_{ii} - a_i) \ \ \ \ \ \ (i = 1,2).$$

Write $p_{12} = (e_{11} - a_1)h_{12}$, $p_{21} = h_{21}$. Then $p_{ij} \in bar(U)_{ij}$, $\overline{p_{ij}} = \overline{f_{ij}}$, and we shall prove
$$p_{ij}p_{ji} = e_{ii} \ \ \ \ \ \ (i \neq j; i,j = 1,2).$$
In fact,  
$$p_{12}p_{21} = [(e_{11} - a_1)h_{12}]h_{21} = (e_{11} - a_{1})(h_{12}h_{21}) = (e_{11} - a_{1})(e_{11} + a_{1}) = e_{11}.$$
But
$$a_{i}h_{ij} = (h_{ij}h_{ji} - e_{ii})h_{ij} = h_{ij}(h_{ji}h_{ij}) - h_{ij} = h_{ij}(e_{jj} + a_{j}) - h_{ij} = h_{ij}a_j,$$
so that $p_{12} = h_{12} - a_1h_{12} = h_{12} - h_{12}a_2 = h_{12}(e_{22} - a_{2})$ and
$p_{21}p_{12} = h_{21}[h_{12}(e_{22} - a_2)] = (h_{21}h_{12})(e_{22} - a_2) = e_{22}.$
Also we have $e_{ij}p_{ji} = p_{ji}e_{ij} = 0$ \ \ \ $(i \neq j; i, j = 1, 2).$
Thus write $v_k = p_{12} + p_{21}$. Then $\overline{v_k} = \overline{f_{12}} + \overline{f_{21}} = \overline{w_k}$, implying $v_k \notin \mathcal{D}$. Also 
$$v_{k}^2 = (p_{12} + p_{21})^2 = e_{11}+e_{22} = 1.$$
Writing $b= \alpha_{1} e_{11} + \alpha_{2} e_{12} + \alpha_{3} e_{21} + \alpha_{4} e_{22}$, we have $\iota(b) = \alpha_{4} e_{11} - \alpha_{2} e_{12} - \alpha_{3} e_{21} + \alpha_{1} e_{22}$,
\begin{eqnarray*}
bv_{k} &=& \alpha_{1} p_{12} + \alpha_{2} e_{12}p_{12} + \alpha_{3} e_{21}p_{21} + \alpha_{4} p_{21}
\\&=& \alpha_{1} p_{12} - \alpha_{2} p_{12}e_{12} - \alpha_{3} p_{21}e_{21} + \alpha_{4} p_{21} = v_k \iota(b)
\end{eqnarray*}
which completes the proof of the lemma.
\end{proof}

\begin{theorem}\label{t31} Let $F$ be an algebraically closed field and $(U,\omega)$ be a finite dimensional b- alternative algebra over $F$ with unity element $1$. Then $U$ has a Wedderburn $b$-decomposition $U=S\oplus V\oplus rad(U)$.
Furthermore, $bar(S)$ is a semisimple algebra and $V\oplus rad(U)$ is a nil ideal of $bar(U)$.
\end{theorem}
\begin{proof}
The same inductive argument based on the dimension of $U$ which is used for associative algebras suffices to reduce the proof of the theorem to the case $rad(U)^2 = 0$.
 Let us take the quotient b-algebra $U/rad(U)$. By \cite[Corollary 3.1]{Guzzo}, we have $rad\big(U/rad(U)\big)=0$ which implies that $U/rad(U)$ is $b$-semisimple, by \cite[Theorem 4.2]{Guzzo}. So, $bar\big(U/rad(U)\big)$ is a sum of
minimal $b$-ideals $I_{1}\oplus \cdots \oplus I_{k} \oplus I_{k+1} \oplus \cdots \oplus I_{s}\oplus J_{s+1}\oplus \cdots \oplus J_{r},$ of $U/rad(U),$ where $I_{i}$ are simple associative algebras $(i=1,\ldots ,k)$, $I_{i}$ are Cayley algebras $(i=k+1,\ldots ,s)$ and $J_{j}^{2}=0$ $(s+1\leq j\leq r)$, by Proposition \ref{p21} and \cite[Corollary 1, page 151]{ZHEV}. Since every ideal $I_{i}$ $(i=1,\ldots ,k)$ is a total matrix algebra $\mathfrak{M}_{t_{i}}$ of degree $t_{i}$ $(i=1, \ldots ,k)$, by \cite[Corollary b, \S  3.5]{Pierce}, then $bar(U)$ contains a direct sum of pairwise orthogonal total matrix algebras $\mathfrak{M}_{i}$ of degree $t_{i}$ $(i=1, \ldots ,k)$ such that $\mathfrak{M}_{1} \oplus \cdots \oplus \mathfrak{M}_{k}\cong \mathfrak{M}_{t_{1}} \oplus \cdots \oplus \mathfrak{M}_{t_{k}}$, by Lemma \ref{l34}. On the other hand, $bar(U)$ contains a vector subspace $V$ such that $V\cong J_{s+1} \oplus \cdots \oplus J_{r}$ and $V^2 \subset rad(U)$, by Lemma \ref{l35}.
Moreover $bar(U)$ contains also a subalgebra $\mathcal{C} \cong I_{1} \oplus \cdots \oplus I_{k}$, by Lemma \ref{lnovo}.

Let us define $S=F1\oplus \mathfrak{M}_{1} \oplus \cdots \oplus \mathfrak{M}_{k} \oplus \mathcal{C}$. Certainly, $S$ is a b-subalgebra of $U$ such that $bar(S)$ is a semisimple algebra and which yields $S$ semisimple, by \cite[Proposition 4.1.]{Couto}. Hence, $S$ is a b-semisimple, by \cite[Lemma 4.1.]{Couto}. Moreover, since $(\mathfrak{M}_{1} \oplus \cdots \oplus \mathfrak{M}_{k})\cap V=(0)$ and $\mathcal{C} \cap (\mathfrak{M}_{1} + \cdots + \mathfrak{M}_{k} + V + rad(U)) = (0)$, then $U=S\oplus V\oplus rad(U)$.

Finally,  let us show that the subspace $V\oplus rad(U)$ is a nil ideal of $bar(U)$. In fact, for arbitrary elements $x\in bar(U)$ and $y\in V\oplus rad(U)$, we have
$\overline{x}=\sum _{i=1}^{k} \overline{a_{i}}+ \sum _{i=k+1}^{s} \overline{a_{i}} + \sum _{j=s+1}^{r} \overline{b_{j}}$, where $\overline{a_{i}}\in I_{i}$ $(i = 1, \ldots,s)$ and $\overline{b_{j}}\in J_{j}$ $(j= s+1, \ldots ,r)$, and $\overline{y}=\sum _{j=s+1}^{r} \overline{c_{j}}$, where $\overline{c_{j}}\in J_{j}$ $(j= s+1, \ldots ,r)$. Hence $\overline{xy}=\overline{x}\,\overline{y}\in J_{s+1}\oplus \cdots \oplus J_{r}$ which implies $xy\in V\oplus rad(U)$. Similarly, we prove $yx\in V\oplus rad(U)$. Thus,  $V\oplus rad(U)$ is an ideal of $bar(U)$. Since $y^{2}\in rad(U)$, then $y$ is a nilpotent element and therefore we can conclude that $V\oplus rad(U)$ is a nil ideal of $bar(U)$.\end{proof}

\begin{theorem}{\bf (Main)} Let $F$ be an algebraically closed field and $(U,\omega)$ a finite dimensional $b-$alternative algebra over $F$. Then $U$ has a Wedderburn $b$-decomposition.
\end{theorem}
\begin{proof} Let us consider a principal idempotent $e$ and take $U=U_{11}\oplus U_{10}\oplus U_{01}\oplus U_{00}$, the Peirce decomposition of $U$, relative to $e$. We know that: {\it (i)} $U_{11}$ is a subalgebra with unity element $e$; {\it (ii)} $U_{10}\oplus U_{01}\oplus U_{00}\subset R(U)$ and {\it (iii)} $R(U)=R\big(U_{11}\big)\oplus U_{10}\oplus U_{01}\oplus U_{00}$. Moreover, as the idempotent $e$ is principal in $U$, then it has weight one. This implies that $U_{11}$ is a $b$-subalgebra of $U$. Thus, $U$ and $U_{11}$ admit the decompositions $U=Fe\oplus bar(U)$ and $U_{11}=Fe\oplus bar\big(U_{11}\big)$, respectively.

From Theorem \ref{t31}, we can decompose $U_{11}$ as a direct sum \\
\centerline{$U_{11}=S \oplus W_{11}\oplus rad\big(U_{11}\big)$,}\\
where $S$ is a b-semisimple b-subalgebra of $U_{11}$ such that $bar(S)$ is a semisimple algebra, $W_{11}$ is a vector subspace of $bar\big(U_{11}\big)$ satisfying $W_{11}^{2}\subset rad\big(U_{11}\big)$ and  $W_{11}\oplus rad\big(U_{11}\big)$ is a nil ideal of $bar\big(U_{11}\big)$. It follows that, $S$ is a b-semisimple b-subalgebra of $U$, by \cite[Proposition 4.1. and Lemma 4.1.]{Couto}.

Now, let us observe that \\
\centerline{$rad(U)\bigcap U_{11}\subset R(U)\bigcap U_{11}= R\big(U_{11}\big)\subset bar\big(U_{11}\big)$,}\\
by \cite[Teorema 4.1.]{Couto} and Lemma \ref{l0}, and $bar\big(U_{11}\big)=bar\big(S\big) \oplus W_{11}\oplus rad\big(U_{11}\big)$. Hence, $rad(U)\bigcap U_{11}\subset W_{11}\oplus rad\big(U_{11}\big)$, because $S\bigcap R\big(U_{11}\big)=\{0\}$. Let us take $V_{11}$ an complementary subspace of $rad(U)\bigcap U_{11}$, in $W_{11}\oplus rad\big(U_{11}\big)$. Then \\
\centerline{$W_{11}\oplus rad\big(U_{11}\big)=V_{11}\oplus \big(rad(U)\bigcap U_{11}\big)$.}\\
Since $rad\big(U_{11}\big)=bar\big(U_{11}\big)^{2}\bigcap R\big(U_{11}\big)\subset bar(U)^{2}\bigcap R(U)=rad(U)$, by \cite[Theorem 4.2.]{Couto}, then $V_{11}^{2}\subset rad(U)$. Thus \\
\centerline{$ U_{11}=S \oplus V_{11}\oplus \big(rad(U)\bigcap U_{11}\big)$,}\\
where $V_{11}^{2}\subset rad(U)$.

Next, let us consider the complementary subspaces $V_{10},~V_{01}$ and $V_{00}$, in $U_{10},~U_{01}$ and $U_{00}$, respectively, such that $U_{10}=V_{10}\oplus rad(U)\cap U_{10}$, $U_{01}=V_{01}\oplus rad(U)\cap U_{01}$ and $U_{00}=V_{00}\oplus rad(U)\cap U_{00}$ and take the vector subspace $V=V_{11}\oplus V_{10}\oplus V_{01}\oplus V_{00}$. Certainly, $V$ is a vector subspace of $bar(U)$ and
\begin{eqnarray*} U&=&U_{11}\oplus U_{10}\oplus U_{01}\oplus U_{00}\\
&=&S \oplus V_{11}\oplus \big(rad(U)\bigcap U_{11}\big)\oplus U_{10}\oplus U_{01}\oplus U_{00}\\
&=&S \oplus V_{11}\oplus \big(rad(U)\bigcap U_{11}\big)\oplus V_{10}\oplus \big(rad(U)\bigcap U_{10}\big)\\
&&\oplus V_{01}\oplus \big(rad(U)\bigcap U_{01}\big)\oplus V_{00}\oplus \big(rad(U)\bigcap U_{00}\big)\\
&=&S \oplus V\oplus rad(U).
\end{eqnarray*}
Now, $V_{11}$ and $V_{10}$ are vector subspaces of $bar(U)$ and $R(U)$ respectively which implies $V_{11}V_{10}\subset \big(bar(U)\big)^{2}$ and $V_{11}V_{10}\subset R(U)$. This yields \\
\centerline{$V_{11}V_{10}\subset R(U)\bigcap \big(bar(U)\big)^{2}=rad(U)$,}\\ by \cite[Teorema 4.2]{Couto}. Similarly, we can show that the products $V_{10}V_{01}$, $V_{10}V_{00}$, $V_{01}V_{11}$, $V_{01}V_{10}$, $V_{00}V_{01}$ and  $V_{00}^{2}$ are subsets of $rad(U)$. As all remaining pro\-ducts are zeros, then we can conclude that $V^{2}\subset rad(U)$.\end{proof}

\section{Final remarks}
Importantly, the concept of heredity to b-algebras can not be extended to alternative algebras, we can see this through an example found in \cite{bruno}.

\end{document}